\documentclass[11pt]{article}

\usepackage{amsmath}
\usepackage{parskip}
\usepackage{amssymb}
\usepackage{amsthm}
\usepackage{latexsym}
\usepackage{color}
\usepackage{mathtools}
\usepackage{graphicx}
\usepackage{appendix}
\usepackage{color} 
\usepackage{enumerate}
\usepackage{diagbox}
\usepackage[T1]{fontenc}
\usepackage{palatino}
\usepackage[english]{babel}
\usepackage[table]{xcolor}
\usepackage{calrsfs}
\usepackage{geometry}
\usepackage{url}
\DeclareSymbolFont{calletters}{OMS}{cmsy}{m}{n}
\DeclareSymbolFontAlphabet{\mathcal}{calletters}
\DeclareMathAlphabet{\mathpzc}{OT1}{pzc}{m}{it}

\def\b*{\begin{eqnarray*}}
\def\e*{\end{eqnarray*}}

\newtheorem{theorem}{Theorem}[section]
\newtheorem{definition}[theorem]{Definition}
\newtheorem{proposition}[theorem]{Proposition}
\newtheorem{lemma}[theorem]{Lemma}
\newtheorem{corollary}[theorem]{Corollary}

\newtheorem{assumption}[theorem]{Assumption}

\newtheorem{remark}[theorem]{Remark}

\newcommand{\No}[1]{\left\|#1\right\|}     

\def \R{\mathbb{R}}

\def \G{\mathbb{G}}

\def\Ac{{\cal A}}

\def\Cc{{\cal C}}

\def\Ec{{\cal E}}

\def\Hc{{\cal H}}

\def\Pc{{\cal P}}

\def\dbE{\mathbb{E}}
\def\dbF{\mathbb{F}}

\def\dbH{\mathbb{H}}

\def\dbP{\mathbb{P}}
\def\dbR{\mathbb{R}}

\def\a{\alpha}
\def\b{\beta}
\def\g{\gamma}

\def\e{\varepsilon}

\def\f{\varphi}

\def\G{\Gamma}

\def\Si{\Sigma}
\def\O{\Omega}

\def\pa{\partial}

\def\ms{\medskip}

\def\cd{\cdot}

\newcommand{\ba}{\begin{array}} 
\newcommand{\ea}{\end{array}}
\newcommand{\be}{\begin{equation}}
\newcommand{\ee}{\end{equation}}
\newcommand{\bea}{\begin{eqnarray}}
\newcommand{\eea}{\end{eqnarray}}
\newcommand{\beaa}{\begin{eqnarray*}}
\newcommand{\eeaa}{\end{eqnarray*}}

\def\q{\quad}

 \title{Principal-Agent Problem with Common Agency without Communication\footnote{This work is supported by the ANR project Pacman, ANR-16-CE05-0027}}

 \author{Thibaut Mastrolia\footnote{CMAP, \'Ecole Polytechnique, Route de Saclay, 91128, Palaiseau, \texttt{thibaut.mastrolia@polytechnique.edu}. This author gratefully acknowledges the Chair Financial Risks (Risk Foundation, sponsored by Soci\'et\'e G\'en\'erale) for financial supports.} \and  Zhenjie Ren\footnote{CEREMADE, Universit\'e Paris-Dauphine, PSL research university, Place du Mar\'echal de Lattre de Tassigny, 75016 Paris, \texttt{zhenjie.ren@ceremade.dauphine.fr}}}
 
             \date{\today}

\begin{document}

 \maketitle

\begin{abstract}
In this paper, we consider a problem of contract theory in which several Principals hire a common Agent and we study the model in the continuous time setting. We show that optimal contracts should satisfy some equilibrium conditions and we reduce the optimisation problem of the Principals to a system of coupled Hamilton-Jacobi-Bellman (HJB) equations. We provide conditions ensuring that for risk-neutral Principals, the system of coupled HJB equations admits a solution. Further, we apply our study in a more specific linear-quadratic model where two interacting Principals hire one common Agent. In this continuous time model, we extend the result of \cite{Bernheim} in which the authors compare the optimal effort of the Agent in a non-cooperative Principals model and that in the aggregate model, by showing that these two optimisations coincide only in the first best case (see Proposition \ref{prop:bernheim}). We also study the sensibility of the optimal effort and the optimal remunerations with respect to appetence parameters and the correlation between the projects (see Proposition \ref{prop:faitstylized}). 

\vspace{5mm}

\noindent{\bf Key words:} Moral hazard models, common agency, system of HJB equations.
\vspace{5mm}

\noindent{\bf AMS 2000 subject classifications:} 93E20, 91A06, 91A15, 91B16, 91B70, 60H30.

\vspace{0.5em}
\noindent{\bf JEL subject classifications:} C73, D21, D82, D86, J41.
\end{abstract}

\tableofcontents

\section{Introduction}
In 2016, the Intergovernmental Science-Policy Platform on Biodiversity and Ecosystem Services (IPBES for short) deeply investigated\footnote{see the report \cite{potts2016summary} approved by the Plenary of IPBES at its fourth session in Kuala Lumpur.} pollinators and pollinisation processes in food production issues. To illustrate the motivation of our study in the light of the report of the IPBES, consider for instance one field divided for planting various crops each of them managed by an owner (\textit{she}). It interests the owners to collaborate with a beekeeper (\textit{he}) to ensure a better yield. In this case, the beekeeper distributes his beehives in the fields for the different crops, and manage to improve, through the pollination, the production of the crop-owners. Therefore, each crop-owner has to provide the beekeeper sufficiently good incentives to take care of her crop. The situation in this example typically illustrates a common agency problem with information asymmetry, that is, the crop-owners need to design contracts to motivate the beekeeper, each hoping of improving her own production, without observing directly how the beekeeper and his bees impact the crops. The crop-owners play thus the role of the Principals and the beekeeper can be seen as an Agent, in the so-called multi-Principals/Agent problem and more generally in the theory of incentives.\vspace{0.5em}  

\noindent The theory of incentives or contract theory appeared in the 70's with the work of Mirrlees \cite{mirrlees1976optimal} among others, by studying optimal payments schedule between two economical entities under imperfect observations on the performances of the employee who manages the employer's wealth. This has been investigated later in the textbook of Laffont and Tirole  \cite{laffont1993theory} with illustrations in regulation systems and the underlying game played between the two parties, and we refer to the survey book of Laffont and Martimort \cite{laffont2009theory} for a lot of relevant examples related to this theory. More precisely, the situation with moral hazard considered in contract theory can be described as follows. The Principal (\textit{she}) wants to hire another entity, namely the Agent (\textit{he}), to manage a project. Also, the Agent has a reservation utility which allows him to accept or reject the contract proposed by the Principal. The main difficulty encountered by the Principal is that she is potentially not well-informed toward the work of her Agent. Therefore, we need to distinguish two types of problems. First, we may assume that the Principal controls both the effort of the Agent and the optimal remuneration proposed to him. This situation fits with a risk-sharing model, called \textit{first-best}. In this case, we only investigate the problem of the Principal who chooses an optimal effort and an optimal remuneration under the reservation utility constrain of the Agent. In most of the situations, the Principal has no access to observe the effort of her agent, but only observe the outcome of his work. The problem of the Principal is thus to design an optimal contract maximizing her utility among all the contracts satisfying the Agent's reservation utility, without observing, not to say controlling, his effort. This more realistic situation is called moral hazard problem or \textit{second-best}. In practice, we identify this game between the Agent and the Principal with a Stackelberg equilibrium, \textit{i.e.} given any fixed contract, the Agent has the corresponding best reaction, and based on these predictable reactions the Principal choose the best contract in order to maximize her utility.\vspace{0.5em}

\noindent The researches made in the 70's related to contract theory mainly treat the problem in a discrete time framework with one or more periods, \textit{i.e.} the effort of the Agent, the outcome, and all other relevant quantity take values in a finite space. H\"olmstrom and Milgrom studied a continuous time model in the seminal paper  \cite{holmstrom1987aggregation} in which they consider a Brownian model for the wealth process, and assume the Principal does the optimization with the exponential utility. This work has then been extended by many other authors. Sch\"attler and Sung in \cite{schattler1993first} provided first-order sufficient conditions to solve the problem.  Sung in \cite{sung1995linearity} investigated this issue when the Agent can control the diffusion of the output. Then, Hellwig and Schmidt in \cite{hellwig2001discrete} made a link between the discrete time and continuous time models. All these papers used sophisticated tools of stochastic control theory such as dynamic programming and martingale approach. For more details, we refer to the survey paper of Sung \cite{sung2001lectures} or the book of Cvitanic and Zhang \cite{cvitanic2013contract}. \vspace{0.3em}

\noindent Recently, new ingredients have been introduced to the contract theory, which allows us to study more general models and obtain tractable solutions. Sannikov studied in \cite{Sannikov} a model in which the Principal provides continuous payments to the Agent and chooses a random retiring time. As a remarkable contribution of the paper, Sannikov observed that the value of  the Agent's  problem (the continuation utility) should be considered as a state variable for the problem of the Principal. Later, Cvitani\'c, Possama\"i and Touzi proposed in \cite{CPT1,CPT2} a very general procedure to solve Principal/Agent problems with lump-sum payments in which the Agent can control both the drift and the volatility of the output. They observe that as an utility maximization problem, the Agent's problem can be reduced\footnote{We refer to the works of El Karoui and Rouge \cite{rouge2000pricing} and Hu, Imkeller and M\"uller \cite{hu2005utility}.} to solving a backward stochastic differential equation (BSDE) introduced in Pardoux and Peng \cite{pardoux1990adapted} (see also El Karoui, Peng and Quenez \cite{el1997backward}), or its fully-nonlinear generalization, namely second-order BSDE (2BSDE) introduced in Soner, Touzi and Zhang \cite{soner2012wellposedness}. The problem of the Principal, given the best reaction of the Agent, is thus a standard stochastic controlled problem with the two state variables: the output process and the continuation utility of the Agent, as Sannikov suggested.

\noindent Extensions of the H\"olmstrom and Milgrom problem are nevertheless not restricted to the single Principal/single Agent model. It was extended to the single Principal/multi-Agents model by Koo, Shim and Sung in \cite{keun2008optimal} then by Elie and Possama\"i in \cite{EliePossamai}. In these works, the Agents aim at finding a Nash equilibrium given contracts proposed by the Principal. Elie and Possama\"i  proved that the Nash equilibrium among the Agents can be characterized as the solution to a multi-dimensional BSDE, and that the problem of the Principal can be reduced to solving a stochastic control problem which takes the continuation utilities of all the Agents as state variables. Elie and Possama\"i, in particular, studied a single Principal/2-Agents example in which the agents compete with each other. It is curious to see in the example that the less ambitious agent may make less effort and leave the job to the other agent. Recently, it was studied in \cite{mastrolia2017moral} that a Planner aims at acting for the welfare of the agents by finding Pareto optima, and the result is compared with that in \cite{EliePossamai} in which the agents reach Nash equilibria. In particular, the sufficient and necessary conditions are provided in \cite{mastrolia2017moral} so that the Pareto optima coincide with the Nash equilibria or that the Planner can  have a higher return than in the classical multi-agents second best case.\vspace{0.5em}

\noindent In this paper, we focus on a multi-Principals/single Agent problem in the continuous  setting. More precisely, we assume in this paper that several Principals aim at hiring one common agent who works simultaneously for all the Principals. This common agency problem echoes the example of the beekeeper and crop-owners presented above, and as far as we know, it has not been studied in the continuous  setting. Nevertheless, in the discrete case, there are several works having investigated this problem. In the 80's, Baron in \cite{Baron} studied a common agency problem involving regulators facing non-localized externalities. He illustrated his study with the example of the Environemental Protection Agency (EPA) and the public utility commission for the control of a non-localized pollution externality. Each entity has its own goal in preventing risks, so the conflicting interests can appear and a cooperation between the regulators does not hold. Baron thus implements cooperative and non-cooperative equilibrium in his particular regulation model. Another interesting application of common agency problem was shown in Braverman and Stiglitz \cite{BravermanStiglitz}, dealing with the sharecropping of the farmers and landlords. Later, Bernheim and Whinston (see \cite{BernheimWhinston, Bernheim}) proposed a general approach for analysing the problem with common agency. They provided a sufficient and necessary condition which characterizes the equilibrium, and proved that a non-cooperative equilibrium between the Principals is efficient (that is, it attains the classical second best level of effort and outcome) only for the first best level of effort. Dixit, Grossman and Helpman have then extended in \cite{dixit1997common} the works of Bernheim and Whinston to more general cases, by considering non-quasilinear utility functions and by characterizing equilibria for the general common agency problem in terms of Nash equilibria. They also discussed the efficiency in the Pareto sense. \vspace{0.5em}

\noindent As mentioned above, all the previous works studied the common agency problem in the discrete time model. As far as we know, this topic has not been explored yet with a diffusion model in the continuous time setting, while such setting-up apparently simplifies the technical argument and produces tractable solutions (as demonstrated in \cite{Sannikov,CPT1,CPT2}). Our work can be treated as an extension of the works of Bernheim and Whinston and that of Dixit, Grossman and Helpman in this direction, and we also try to characterize the Nash equilibria among the contracts. Our study remains part of the analysis of Nash equilibrium with asymmetry of information (see for instance  \cite{CetinDanilova}), by showing that the problems of each Principal to find such equilibrium is reduced to find a solution to a fully coupled HJB equation. This relation was well investigated in the books  \cite{Carmona, Dockner} and was also applied by Carmona and Yang to a predator trading with looser time constraints in \cite{CarmonaYang}. Here, we register our model in the continuation of all the mentioned works by assuming that each Principal is rational and aims at finding his best reaction remuneration facing to the actions of other Principals which naturally leads to the investigation of Nash equilibria.\vspace{0.5em}

\noindent The general structure of our paper is the following, in Section \ref{soussection:principalequilibrium:intro} we state the common agency model that we study. In particular, we introduce the definition of equilibrium. We then solve the problem of the common agent in Section \ref{section:agentpb} by using a verification result. At the heart of this work, we propose in Section \ref{sec:Principal prob}  a general procedure to solve the common agency problem of the Principals, by providing a characterization of the Nash equilibria in a particular set of contracts and by deriving the corresponding fully coupled system of HJB equations. We prove in particular that when the Principals are risk-neutral, we can find a solution to the corresponding system by considering the solution of the problem in which the Principals are aggregated into a unique Principal. As an application of our work, we study in Section \ref{section:2principaux} a model with two Principals hiring a common Agent by comparing a competitive model in which the Principals are not cooperative to a situation in which the offers of the principals are aggregated and the aggregated offer fits the classical bilateral Principal/Agent model. All along this study, we try to give economical interpretations of our mathematical model and results. 

\noindent Finally, by going beyond the example of beekeepers remunerated by crop-owners introduced at the beginning, this paper could be a good starting point to treat the applications such as government management in which different level of a same Minister compensated a firm, to recall the example in \cite{BernheimWhinston} or foreign borrowing issues as those exposed in the work of Tirole \cite{tirole2003inefficient}.

\paragraph{Notations.} Thorough this paper, $T>0$ is the maturity of the contract and is assumed to be fixed. Let $N$ be a positive integer. We denote by $|\cdot|$ the absolute value and $\| \cdot \|$ the norm in $\mathbb R^N$ to simplify the notations. We identify $\mathbb R^{N,N}$ with the space of real square matrices with size $N$, endowed with the Euclidean norm on $\R^{N,N}$. Let $M$ be in $\R^{N,N}$, $1\leq j\leq N$ and $1\leq \ell\leq N$, we denote by $M^{j,:}\in \R^{1\times N}$ (resp. $M^{:,\ell}\in \R^{N,1}$) its $j$-th row (resp. its $\ell$-th column). We set $M^\top\in \R^{N, N}$ the transpose of $M$. We also identify $\R^N$ with $\R^{1,N}$ and for any element $x$ of $\mathbb R^N$, we denote by $x^i$ its $i$th component and $x^{(-i)}$ an $\mathbb R^{N-1}$ dimensional vector such that $x^{(-i)}:= (x^1,\dots, x^{i-1}, x^{i+1},\dots, x^N)^\top$.

\noindent Let $G:\mathfrak b\in \mathbb R^N\longmapsto G(\mathfrak b)\in \mathbb R$, we denote by 
$\nabla_{\mathfrak b} G(\mathfrak b)\in \mathbb R^N$ the gradient of $G$. When $G$ is multivariate and when there is no ambiguity on the notations, $\nabla_{\mathfrak b} G$ denotes the gradient of $G$ with respect to the vector $\mathfrak b$. When $N=1$ we simply write $\partial_y v$ as the derivative of any function $v$ with respect to the real variable $y$. Assume that a real function $u$ is multivariate with $x\in \mathbb R^N$ and $y\in \mathbb R$, we denote by $\partial_{y,y}^2u$ the second derivative of $u$, $\nabla^2_x u$ the Jacobian matrix with size $N\times N$ of $u$ with respect to $x$ and $\partial^2_{x,y}u$ the $\mathbb R^{N,N}$-valued diagonal matrix of the derivatives of $u$ with respect to $y$ and $x$. We denote by $\mathcal C^{1,2,2}([0,T]\times \mathbb R^N\times \mathbb R)$ the space of functions continuously differentiable in time and twice continuously differentiable functions with respect to their space variables, and we simply the notation by writing $\mathcal C^{1,2,2}$ when there is no ambiguity on the underlying spaces.

\section{The multi-Principal/Agent model}\label{soussection:principalequilibrium:intro}
In this paper we consider the Principals-Agent problem in which $N$ Principals hire an Agent in order to manage $N$ risky projects. In this model, we assume that the projects are possibly correlated among each others and the action of the hired Agent impacts on every project. This section introduces the mathematical model of the problem. We first define the dynamics of the $N$ projects, and then set the utility maximization problems of both the Agent and the Principals as a Stackelberg equilibrium.

Let $\O:=C([0,T],\dbR^N)$ be the canonical space, and $X$ be the canonical process. In our model, the process $X$ describes the outputs of all $N$ projects. Let $\dbF$ be the canonical filtration, for any finite dimensional normed space $(E,\No{\cdot}_E)$, $ \mathcal P(E)$ (resp. $\mathcal P_r(E)$) will denote the set of $E-$valued, $\mathbb F-$adapted integrable processes (resp. $\mathbb F-$predictable processes).

\noindent Let $\Pc$ be the set of all probability measures on $\O$. In the paper, we consider the models on $X$ in the following form:
\bea\label{eq:output}
& (\nu, \dbP^\nu)\q\mbox{such that\q $\nu$ is $\dbF$-progressive, $\dbP^\nu\in \Pc$ and}&\notag\\
& X_t = x + \int_0^t b(s, X_s, \nu_s )ds +  \int_0^t \Sigma_s dW^{\dbP^\nu}_s,\q\dbP^\nu\mbox{-a.s.}&
\eea
where 
\begin{itemize}
\item $x\in \R^N$ is the initial value of the output;
\item  the drift function $ b: [0,T]\times \mathbb R^N\times \mathbb R^N\longrightarrow\R^N$ is of components denoted by $b^i$, i.e. $b(t,x,\nu)=(b^i(t,x,\nu^i))^\top_{1\leq i\leq N }$;
\item  $\Sigma$ is $\dbF$-progressively measurable process taking values in $\mathbb R^{N, N}$, characterizing the correlation of noise among the different projects;
\item $W^{\dbP^\nu}$ is a $\dbP^\nu$-Brownian motion.
\end{itemize}
The process $\nu$ is regarded as the effort of the Agent, and influences the output $X$ through the drift function $b$. More exactly, we set the following assumptions on the drift and the volatility of the output process.

\ms

\begin{assumption}\label{assum:si b}
$\Sigma$ is bounded and for any $t\in [0,T]$, the matrix $\Sigma_t$ is invertible with bounded inverse. The drift function $b(t,X_t,\nu)$ is $\dbF$-progressive for each $\nu \in \dbR^N$. For any $i\in \{ 1,\dots,N\}$ and  every $(t,x)\in [0,T]\times \mathbb R^N$, the map $\nu\in \mathbb R^N\mapsto b^i(t,x,\nu^i)$ is continuously differentiable and there exists a positive constant $C$ such that for any $(t,x,\nu)\in [0,T]\times \mathbb R^N\times \mathbb R^N$
$$ |b^i(t,x, \nu^i)|\leq C(1+\|x\|+|\nu^i|),\quad \|\nabla_\nu b(t,x, \nu)\|\leq C.$$
\end{assumption}

\noindent For technical reasons, we constrain our analysis on the following admissible effort:
\begin{definition}[admissible effort]
We denote by $\mathcal A$ the set of admissible efforts $(\nu,\dbP^\nu)$ such that
\begin{multline*}
\mbox{\eqref{eq:output} holds true, $\nu$ takes values in  $A\subset \dbR^n$ so that:}\\
 (M_t^\nu)_{t\in [0, T]}:= \left(\mathcal E\left(\int_0^t b(s,X_s, \nu_s)\cdot \Sigma^{-1}_s dW^{\dbP^\nu}_s \right) \right)_{t\in [0,T]}\mbox{ is an  $\mathbb F$-martingale,}
\end{multline*}
where $\Ec$ denotes the Dol\'eans-Dade exponential.
\end{definition}
Since $M^\nu$ is an $\dbF$-martingale for each admissible effort, $\{\dbP^\nu\}_\nu$ are all equivalent measures. More precisely, we have
\bea\label{P0}
\frac{d\mathbb P^\nu}{d\mathbb P}= M^\nu_T,\q\mbox{where}\q\dbP:=\dbP_0\circ \Big(x+\int_0^\cd \Si_sdX_s\Big)^{-1}
\eea
and $\dbP_0$ is the Wiener's measure.

In our model, the Agent derives a utility given the salaries from the $N$ Principals at the terminal time $T$, and his effort diminishes his general payoff through the cost function $c: [0,T]\times \mathbb R^N\times \mathbb R^N \longmapsto \mathbb R_+$. More precisely, each Principal, e.g. Principal $i$, proposes to the Agent a contract with the salary denoted by $\xi^i$, a $\mathbb R$-valued $\mathcal F_T$-measurable random variable. The total salary for the Agent is thus $\xi\cdot \mathbf 1_N$, where $\xi=(\xi^1,\dots,\xi^N)^\top$ and $\mathbf 1_N$ denotes the $N$-dimensional vector $(1,\dots,1)^\top$, living in some space $\mathcal C_N$ of admissible contracts defined below. We consider the exponential utility for the Agent with risk aversion parameter $R_A>0$, so that the utility of the Agent at time $t=0$ given a set of salary $\xi$ and an admissible effort $(\nu,\dbP^\nu)\in \Ac$ is
\begin{equation*}
u_0^A(\xi,\nu):= \mathbb E^{\mathbb P^\nu}\left[-\exp\left(-R_A\left( \xi\cdot \mathbf 1_N -\int_0^T c(t,X_t,\nu_t)dt\right)\right) \right].
\end{equation*}

\begin{assumption}\label{assum:c}
The cost function $c(t,X_t,\nu)$ is $\dbF$-progressive for each $\nu\in \dbR^N$. For any $(t,x)\in [0,T]\times \R^N$, the map $\nu\longmapsto c(t,x,\nu)$ is continuously differentiable, increasing and convex. There exist $0< \underline \kappa<\kappa$  and $(m,\underline m)\in [1,+\infty)\times (0,m]$ such that 
$$ 0\leq c(t,x,\nu)\leq C(1+\|x\|+\|\nu\|^{1+m}), $$
$$\underline\kappa  \|\nu\|^{\underline m} \leq \|\nabla_\nu c(t,x,\nu)\|\leq \kappa  \Bigg(1+\|\nu\|^{m}\Bigg) \, \quad \mbox{and} \quad \overline{\lim}_{\|\nu\|\to +\infty} \frac{c(t,x,\nu)}{\|\nu\|}=+\infty.$$
\end{assumption}

The cost function $c$ depends obviously on the effort of the Agent and the fact that $c$ is increasing in it emphasizes that making a bigger effort leads to a bigger cost for the Agent. The convexity assumption can be seen as an exhaustion effect, \textit{i.e.} higher is the work of the Agent, higher is his      sensitivity to increase or decrease his effort. The cost $c$ can also depends on the outcome itself. For instance, if the value of the outcome is low, the Agent could be somehow depressed and his action could be directly impacted by it.\vspace{0.3em}

The Agent tries to maximize his utility given a panel of salaries $\xi$, by choosing an admissible effort. Therefore, given a panel of contracts $\xi$ proposed by the $N$ Principals, the weak formulation of the {\bf Agent's problem} is 
\begin{equation}\label{pbAgent}
U_0^A(\xi):= \sup_{(\nu,\dbP^\nu)\in \mathcal A} u_0^A(\xi,\nu).
\end{equation}

Denote by $\mathcal A^\star(\xi)$ the set of best responses of the Agent, given a vector of salaries $\xi$. Given the best response, we will study Nash equilibriums among the $N$-Principals when they try to maximize their own utilities.

\noindent Consider an Agent with a reservation utility $R_0\in \mathbb R$, that is, the $N$-Principals have to solve their utility maximization problem under the constraint that the set of contracts $\xi$ proposed to the Agent satisfies
\begin{equation}
U_0^A(\xi)\geq R_0.
\end{equation} 
Otherwise, the Agent may refuse to work for them.
Now we introduce the problem of the $i$th Principal. Assume that her payoff is given by a function $U_{P^i}:\mathbb R\longrightarrow\mathbb R$ increasing and concave. The {\bf problem of Principal $i$} is  
 \begin{equation}\label{pbprincipali}
 U_0^{P_i}:= {\sup_{\xi^i\in \mathcal C^i: U_0^A(\xi)\geq R_0}} \sup_{\nu^\star\in \Ac^\star(\xi)} \mathbb E^{\nu^{\star}}\left[ \mathcal K_0^T U_{P_i}(\ell_i(X_T)-\xi^i)\right],
 \end{equation}
 where
 \begin{itemize}
 \item $\mathcal K_0^T:= e^{-\int_0^T k_r dr}$ is the discounting process with a bounded process $k$;
 \item  $\ell_i:\mathbb R^N\longrightarrow \mathbb R$ is a liquidation function of linear growth;
 \item $\Cc^i$ is the set of admissible contract which will be defined later (see \eqref{def:admissiblecontract}).
 \end{itemize}
 \ms

\begin{remark}\label{rem:Asingleton}
Our definition of the Principal's problem follows the one in \cite{CPT2}. As in the classic literature (see e.g. \cite{Sannikov, CPT1,CPT2}), we will use a semimartingale representation of the contract $\xi$ to transfer the Principal's problem to a standard stochastic control problem. After this transformation, we shall see that the maximization over $\nu^\star\in \Ac^\star(\xi)$ does not change the structure of the problem. Therefore, for the simplicity of notations, we assume in the rest of the paper that $\Ac^\star(\xi)$ is a singleton, and the problem of Principal $i$ can read as
\beaa
 U_0^{P_i}:= {\sup_{\xi^i\in \mathcal C: U_0^A(\xi)\geq R_0}}  \mathbb E^{\nu^{\star}(\xi)}\left[ \mathcal K_0^T U_{P_i}(\ell_i(X_T)-\xi^i)\right].
\eeaa
\end{remark}
 

 \noindent We can provide two interesting examples which have been investigated in the discrete case by \cite{BernheimWhinston}.
 
\paragraph{Competitive Principals.} {\it If one considers now that the $i$th Principal receives $X_T^i$ at time $T$ and possibly gets a higher utility if her project is higher than the empirical mean of the others, we have $$\ell_i(x)= x^i+ \gamma^i\left(x^i- \frac{1}{N-1}\sum_{i\neq j} x^j\right),\; x\in \mathbb R^N,$$ where $\gamma^i\geq 0$ is typically her appetence parameter toward project $i$. This situation fits exactly with \textit{noncooperative Principals' behaviors} investigated in \cite{BernheimWhinston}. We refer to Sections \ref{section:impactEffortAmbitions} for an example of the impact of the different parameters in this situation.}

\paragraph{Aggregated offer.} {\it If one considers now that one can aggregate the $N$-Principals, then we reduce this problem to a single Principal-Agent model in the classical second-best case in which each $i$th component of the output has an efficiency parameter $\gamma^i$ compared to the other components. In this case, the payoff of the aggregated Principal, called the \textit{parent firm} is
$$\sum_{i=1}^N X_T^i\left(1+\gamma^i-\frac{1}{N-1}\sum_{j\neq i}\gamma^j\right)-\xi$$ where $\xi$ is the aggregation of the $N$ salaries. This case coincides with the cooperative model of Section 2 in \cite{BernheimWhinston}, and it will be studied more deeply in Section \ref{section:casPrincipauxFullyCooperatif}.  }

\noindent We finish this section, by introducing the Nash equilibrium among the $N$ Principals.

\begin{definition}[Nash equilibrium]\label{def:nash}
A contract $\xi^\star \in \mathcal C_N$ is a Nash equilibrium for the $N$ Principals if it satisfies
 $$\sup_{\xi^i\in \mathcal C^i} \mathbb E^{\nu^{\star}(\xi^\star)}\left[ \mathcal K_0^T U_{P_i}\left(\ell_i(X_T)-\xi^i \right)\right]= \mathbb E^{\nu^{\star}(\xi^\star)}\left[\mathcal K_0^T U_{P_i}\left(\ell_i(X_T)-\xi^{i,\star}\right)\right],\; 1\leq i\leq N$$
 with $U_0^A(\xi^\star)\geq R_0. $
\end{definition}

\section{Solving the Agent problem}\label{section:agentpb}

The essential idea in \cite{Sannikov, CPT1,CPT2} is to study the contracts with a (forward) semimartingale representation. The representation, on the one hand, can help solve the Agent problem through simple verification argument. On the other hand, it reduces the Principal problem to standard stochastic control problem. In our study, we shall borrow this idea. However, instead of representing every contract $\xi^i$, we only use the semimartingale representation of the sum of the salaries $\xi\cd 1_N$.

First, let us recall the definition of BMO space.
 \begin{definition}\label{bmo}
 For any process $Z\in \mathbb H^2:=\{Z\in \mathcal P_r(\mathbb R^N): \dbE^\dbP[\int_0^T \|Z_t\|^2 dt]<\infty\}$ (recall $\dbP$ defined in \eqref{P0}), we say that $Z\in \dbH^2_{BMO}$, if there exists a non negative constant $C$ such that for any $t \leq T$ we have
 $$\mathbb E^\dbP_t\left[\int_t^T \|Z_s\|^2ds \right]\leq C^2.$$
  \end{definition}

In the rest of paper, we constrain our study on the following type of set of contracts:
\begin{multline}\label{def:CN}
 \mathcal C_N:=\Big\{ \xi:~ U_0^A(\xi)\geq R_0\q\mbox{and there is $Y_0\in \dbR^N$ and $Z\in \dbH^2_{BMO}$ such that}\q\\
 \xi\cdot\mathbf 1_N = Y_0 + \int_0^T  G(t,X_t, Z_t) dt +\int_0^T Z_t \cdot  dX_t,\q\dbP\mbox{-a.s.}
  \Big\},
\end{multline}
where $G(s,x,z):= \frac{R_A}2  \|  \Sigma_s^\top z\|^2- \sup_{\nu\in A}\big(b(t, x,\nu)\cdot z- c(t,x, \nu)\big)$. By recalling Remark \ref{rem:Asingleton}, we assume here that  $\arg\max_\nu\big(b(t, x,\nu)\cdot z- c(t,x, \nu)\big)$ is a singleton for each $(t,x,z)$ and denote
\bea\label{vstar}
\nu^\star_t(x,z):=\arg\max_\nu\big(b(t, x,\nu)\cdot z- c(t,x, \nu)\big).
\eea

\begin{lemma} Let Assumptions \ref{assum:si b} and \ref{assum:c} be true. Then
\beaa
\|\nu_t^\star(x,z)\| \leq C\left(1+\|z\|^{\frac{1}{\underline m}}\right)\q\mbox{and}\q
|G(t,x,z)|\leq C\left(1+  \|z\|^2+ \|z\|\|x\|\right). 
\eeaa
\end{lemma}
\begin{proof}
The proof follows the same line that the proofs of \cite{EliePossamai} or \cite{EMP}.
\end{proof}

\begin{remark}
Thanks to the work on the backward SDE with quadratic growth, we know that the set $\Cc_N$ is quite large. In fact, provided that the  BSDE with terminal value $\xi$ and generator $G$ has a solution $(Y,Z)$ with $Z$ in the BMO space, and that  $U_0^A(\xi)\geq R_0$, then $\xi$ belongs to $\mathcal C_N$. For example, due to \cite{Kobylanski}, it is well-known that all bounded $\xi$ such that $U_0^A(\xi)\geq R_0$ belong to $\Cc_N$.\end{remark}
The following proposition reveals the optimal effort of the Agent.

\ms
\begin{proposition}\label{prop:agentpb}
Let Assumptions \ref{assum:si b} and \ref{assum:c} be true. For $\xi\in \Cc_N$, if $(\nu^*,\dbP^{\nu^*})\in \Ac$ where $\nu^*_t:=\nu^\star_t(X_t,Z_t)$, then  we have
\beaa
U^A_0(\xi) = -e^{-R_A Y_0}\q
\mbox{and \q $\nu^*_t$~ is the optimal effort of the Agent.}
\eeaa
\end{proposition}

\begin{proof}
Note that $\dbP$ and $\dbP^\nu$ are equivalent measures for all admissible efforts. Using the semimartingale representation of $\xi\cd 1_N$, we have for admissible $\nu$ that
\beaa
u^A_0(\xi,\nu) &=& \dbE^{\dbP^\nu}\Big[ -\exp\Big(-R_A\big( \xi\cdot \mathbf 1_N -\int_0^T c(t,X_t,\nu_t)dt\big)\Big)\Big]\\
&=& \dbE^{\dbP^\nu}\Big[ -\exp\Big(-R_A\big( Y_0 + \int_0^T  \big(G(t,X_t, Z_t)-c(t,X_t,\nu_t)\big) dt +\int_0^T Z_t \cdot  dX_t \big)\Big)\Big]\\
&\le & \dbE^{\dbP^\nu}\Big[ -\exp\Big(-R_A\big( Y_0 + \int_0^T  \big( \frac{R_A}2 \| \Sigma_t^\top Z_t \|^2- b(t, X_t,\nu_t)\cdot Z_t\big) dt +\int_t^T Z_t \cdot  dX_t\big)\Big)\Big]\\
& =& \dbE^{\dbP^\nu}\Big[ -\exp\Big(-R_A\big( Y_0 + \int_0^T   \frac{R_A}2 \| \Sigma_t^\top Z_t \|^2 dt +\int_t^T Z_t \cdot  \Si dW^{\dbP^\nu}_t\big)\Big)\Big]\\
& = & -e^{-R_A Y_0}.
\eeaa
The last equality is due to the fact that $Z\in \dbH^2_{BMO}$, and the inequality is an equality if $\nu_t = \nu^*_t(X_t,Z_t)$.
\end{proof}

\begin{remark}
One sufficient condition for $(\nu^*,\dbP^{\nu^*})\in\Ac$ can be $\underline{m}\ge 1$, so that $\nu^*\in \dbH^2_{BMO}$ and thus $(\nu^*,\dbP^{\nu^*})\in\Ac$.
\end{remark}

\section{Solving the Principals problem}\label{sec:Principal prob}

As mentioned before, the main difficulty here is that we do not have a semimartingale representation for the component $\xi^i$, when solving the Agent problem. Instead, we only have a such representation for the sum of the salaries, $\xi\cdot {\bf 1}_N$. 

In this paper, we are only interested in Nash equilibriums in the following sets of contracts:
\begin{align*}
\widetilde{\mathcal C}_N&:= \Bigg\{\xi\in \mathcal C_N:\; \exists (y,\alpha,\beta)\in \dbR^N\times \mathcal P(\mathbb R^N)\times \mathbb H^2_{BMO}(\mathbb R^{N,N}),\\
&\hspace{2.5em} \xi^i=y^i+\int_0^T \alpha_s^ids+\int_0^T \beta_s^{:,i}\cdot dX_s,\, \forall 1\leq i\leq N \Bigg\},
\end{align*}
where  $\mathbb H^2_{BMO}(\mathbb R^{N,N})$ extends Definition \ref{bmo} to $\mathbb R^{N,N}$ valued process.

\noindent More intuitively speaking, we search for Nash equilibriums among the contracts with semimartingale form. Therefore, at such a Nash equilibrium, each Principal $i$ should maximize her utility among the following contracts:
\bea\label{def:admissiblecontract}
\Cc^i(\xi^{-i}) ~:=~ \{\xi^i: \xi\in \tilde\Cc_N\},
\eea
where $\xi^{-i}=(\xi^1,\cdots,\xi^{i-1},\xi^{i+1},\cdots,\xi^{N} )^\top$. Without ambiguity, we will only denote the set by $\Cc^i$. 

Note that, similarly to \cite[$(i')$]{BernheimWhinston}, the problem of the $i$th Principal \eqref{pbprincipali} can be rewritten given the contracts $\xi^j$ of the other Principals $j\neq i$
  
   \begin{equation}\label{pbPrincipaliBis}
\sup_{\xi^i\in \mathcal C^i} \mathbb E^{\nu^{\star}(\xi)}\left[ \mathcal K_0^T U_{P_i}\left(\ell_i(X_T)-\xi\cdot \mathbf 1_N + \sum_{j\neq i} \xi^j\right)\right].
 \end{equation}

In the following sections, we try to characterize the Nash equilibriums, guided by the following ideas.

\begin{itemize}

\item The definition in \eqref{def:admissiblecontract} immediately leads to a condition for equilibrium.

\item For each principal, the optimization \eqref{pbPrincipaliBis} can be transformed to a standard stochastic control problem, which leads to an HJB equation. Therefore, we naturally obtain a system of $N$ coupled HJB equations associated with the problems of Principals.
Using classical verification argument, we can show that this system of HJB equations provides a Nash equilibrium for the Principals's problems in the class of contracts.

\item In particular, if all the Principals are risk-neutral, the stochastic control problem for each Principal is associated to a semilinear HJB equation, and thus enjoys a better solvability.
\end{itemize}
\ms

\begin{remark}
We would like to comment and explain why this class of contract seems to be reasonable in our model. First of all, this class of contract is non-Markovian and it is well known that in the single Principal-Agent model, as soon as one assume that the dynamic of the output is boosted by some parameter $\alpha$, \textit{i.e.} $b(t,x,a):= \alpha x+a_t$, one gets non-Markovian contracts. Thus, there is no reason to consider a subset of $\widetilde{\mathcal C}_N$ only containing functions of $X_T$.\vspace{0.3em}

\noindent The second remark is that one could think that since any contract $\xi^i$ is an $\mathcal F_T-$measurable random variable with nice integrability property, by considering $Y_t:= \mathbb E[\xi^i|\mathcal F_t]$, we get from the martingale representation theorem that there exists some process $Z$ such that $\xi^i= \mathbb E[\xi^i] +\int_0^T Z_t\cdot dW_s$. Thus, $\xi^i$ can be seen as the terminal value of a stochastic integral. However, we will show below that due to the Stackelberg game between the Agent and the Principals, this class of contracts is not robust and does not contain a Nash equilibrium (see  Proposition \ref{cor:necessary} below).\vspace{0.5em}

\noindent Finally we would like to insist on the fact that this restriction on components $\xi^i$ is allowed for our work because our aim is to get only the existence of a Nash equilibrium. An extension of this work could be to consider a larger class of contracts included in $\mathcal C_N$ such that the components $\xi^i$ have not necessarily the semi-martingale decomposition with $\alpha$ and $\beta$ and compare the two classes (in term of Pareto efficiency for instance). This question seems however very difficult and is not at the heart of our work.\end{remark}

\noindent Before going further let us explain more the meaning of the class of contracts $\widetilde{\mathcal C}_N$. Any salary $\xi^i$ given by the $i$th Principal depends on the projects of other Principal $X^j,\; j\neq i$. We provide some interpretations of this phenomenon. 

\paragraph*{Interpretation of the model}
{\it First, we can assume that the $i$th Principal promises to the Agent a part of the value of her firm but also a part of the outcome of an other Principal. Second, some externalities can appear as a \textit{network effect} (see for instance \cite{laffont1989externalities, buchanan1962externality, liebowitz1994network} for more explanations with definition of this phenomenon). This fits typically with the introductive example with the beekeeper and crop-owners, since each crop-owner can benefit from the beehives in the other crops through the pollinisation process. An other interpretation is to consider a parent firm and some subsidiaries (the $N$ Principals) who employed the same Agent. For instance and in concrete terms, this model is quite suitable in the example of the German firm Sonnen GmbH in which Agents are connected to produce, use and share energy. }

\noindent \begin{remark}\label{remark:freerider}The case in which a Principal has a free-rider behaviour by proposing no remuneration to the Agent and the other Principal compensates fully the Agent is not an equilibrium due to the common agency framework. Indeed, as showed in Theorem \ref{thm:principal:nash} below, by taking $\alpha^{i,\star}=0$, $\beta^{i,\star}$ is not equal to zero.
\end{remark}

\subsection{Equilibrium conditions}

\begin{proposition}[Necessary condition of Nash equilibrium] \label{cor:necessary}
Let $\xi^\star=(\xi^{i,\star})_{1\leq i\leq N}$ be a Nash equilibrium in $\widetilde C_N$ such that $\xi^{i,\star}$ is characterized by the triplet $(y^i,\alpha^{i,\star},\beta^{i,\star})$. Then, the following equilibrium condition holds for $(y^i,\alpha^{i,\star},\beta^{i,\star})_{1\leq i\leq N}$:
\begin{equation}\label{eqsystem}
\begin{cases}
\displaystyle \sum_{i=1}^N y^i=Y_0,\\[0.5em]
\displaystyle\sum_{i=1}^N\alpha^i_s=G(s,X_s, \sum_{i=1}^N \beta_s^i )\\[0.5em]
\displaystyle \sum_{i=1}^N \beta_s^i= Z_s.
\end{cases}
\end{equation}
\end{proposition}

\begin{proof}
The result simply follows from the definition of $\Cc_N$ \eqref{def:CN} and that of $\Cc^i$ \eqref{def:admissiblecontract}.
\end{proof}

\subsection{General case with fully-nonlinear HJB equations} 

We are going to fully characterize a Nash equilibrium. We start from the problem of the $i$th Principal \eqref{pbPrincipaliBis}. Denote $r_0 := - \frac{\ln (-R_0)}{R_A}$, so that we have
\beaa
U^A_0(\xi) \ge R_0 \iff Y_0 \ge r_0,
\eeaa
thanks to Proposition \ref{prop:agentpb}.
It follows from Proposition \ref{cor:necessary} that 
\beaa
U_0^{P_i}(x)
& = & \sup_{\xi^i\in \mathcal C^i} \mathbb E^{\nu^{\star}}\left[ \mathcal K_0^T U_{P_i}\left(\ell_i(X_T)-\xi\cdot \mathbf 1_N + \sum_{j\neq i} \xi^j\right)\right] \\
& =&\sup_{y^i\geq r_0-\sum_{j\neq i} y^j} \underset{\sum_{i=1}^N\alpha^i_s=G(s,X_s, \sum_{i=1}^N \beta_s^i)} {\sup_{(\alpha^i,\beta^i)\in \mathcal P(\mathbb R)\times \dbH^2_{BMO}}}\mathbb E^{\nu^{\star}}\left[ \mathcal K_0^T U_{P_i}(\ell_i(X_T)-Y_T^{y^i,\alpha^i,\beta^i})\right]\\
&=&\sup_{y^i\geq r_0-\sum_{j\neq i} y^j} \, u^{i}(0,x,y^i),\\
&=&  u^{i}\big(0,x,r_0-\sum_{j\neq i} y^j\big),
\eeaa
 where
\begin{equation}\label{pb:prncipaliyi}
u^{i}(0,x,y^i):= \sup_{\beta^i\in   \dbH^2_{BMO}}\mathbb E^{\nu^{\star}}\left[\mathcal K_0^T  U_{P_i}\left(\ell(X_T)-Y_T^{y^i,\alpha^i,\beta^i}\right)\right],
\end{equation} 
and
 \begin{equation}\label{Ydynamic}
 Y_T^{y^i,\alpha^i,\beta^i}= y^i+\int_0^T \left(G\left(s,X_s, S_{\beta^{(-i)}_s}+ \beta_s^i\right)-S_{\alpha_s^{(-i)}}\right) ds+\int_0^T \beta_s^{i}\cdot dX_s, 
 \end{equation}
 where $S_{\alpha^{(-i)}}:=\sum_{j\neq i}\alpha^j$ and $S_{\beta^{(-i)}}:=\sum_{j\neq i}\beta^j$ and with fixed $(\alpha_j,\beta_j)_{j\neq i}$. 
 
\noindent The problem of the Principal \eqref{pb:prncipaliyi} coincides with a stochastic control problem with the following characteristics:
\begin{itemize}
\item two state variables: the output $X$ and the value process $Y^{y^i, \alpha^i,\beta^i}$;

\item one control variable: the coefficient $\beta^i$.
\end{itemize}
It is associated with the following HJB equation :
\begin{equation}\label{hjbi}
\begin{cases}
\displaystyle -(\partial_t u^i- ku^i)(t,x,y)-\sup_{\beta^i \in \mathbb R^N} H(t,x,y,\nabla_x u^i,\partial_y u^i,\nabla_x^2 u^i,\partial^2_{yy}u^i,\partial^2_{x,y} u^i,S_{\alpha_s^{(-i)}},S_{\beta_s^{(-i)}},\beta^i)= 0,\\[0.8em]
\displaystyle u^i(T,x,y)=U_{P^i}(\ell_i(x)-y),\; (t,x,y)\in [0,T)\times \mathbb R^N\times\mathbb R.
\end{cases} 
\end{equation}
with Hamiltonian $H$ defined for any $(t,x,p,\tilde p,q,\tilde q,r,s_a,s_b,\mathfrak b)\in [0,T]\times \mathbb R^N\times \mathbb R^N\times \mathbb R\times\mathbb R^{N,N}\times \mathbb R\times \mathbb R^{N,N}\times \mathbb R\times \mathbb R^N\times \mathbb R^N $ by
\begin{multline*}
H(t,x,p,\tilde p, q,\tilde q,r,s_a,s_b,\mathfrak b)
\\
:=p\cdot b(t,x, \nu^\star(x,\mathfrak b+s_b)) +\tilde p \Big( G(s,x,\mathfrak b+s_b)-s_a+ \mathfrak b\cdot b(t,x, \nu^\star(x,\mathfrak b+s_b))\Big)\\
+\frac12 \text{Tr}(\Sigma_t\Sigma_{t}^\top q) + \frac12 \text{Tr}(\mathfrak b^\top \Sigma_t \Sigma_t^\top\mathfrak b\tilde q)+\text{Tr}(\mathbf 1_N^\top\Sigma_t\Sigma_{t}^\top\mathfrak b r)
\end{multline*}

\begin{assumption}[$\textbf{FOC}$]\label{firstorder}
There exists a maximizer $\beta^{i,\star}$ for $\mathfrak b\longmapsto H(t,x,p,\tilde p, q,\tilde q,r,s_a,s_b,\mathfrak b)$ satisfying some first order condition given by the relation
\begin{equation}\label{FOC}
\lambda^i(t,x,p,\tilde p, \tilde q,r,s_b, \beta^{i,\star}_t)
~:=~ \nabla_{\mathfrak b} H(t,x,p,\tilde p, q,\tilde q,r,s_a,s_b,\mathfrak b)
~=~0.
\end{equation}
\end{assumption}

\noindent Then, the problem of the $i$th Principal can be solved by a classical verification theorem.\vspace{0.5em}

\begin{proposition}[Verification]\label{prop:Pi} 
Let $\xi\in \widetilde {\mathcal C}_N$ and denote by $(y^j,\alpha^j,\beta^j)$ the triplet associated with $\xi^j, \, j\neq i$. Assume that there exists a solution $u^i$ to HJB equation \eqref{hjbi} in $\mathcal C^{1,2,2}([0,T]\times \mathbb R^N\times \mathbb R)$. Then, by denoting $\beta^{i,\star}_s$ the corresponding maximizer, we deduce that $U_0^{P_i}(x)= u^i(0,x,r_0 -\sum_{j\neq i}y^j)$ and the optimal contract $\xi^{i,\star}$ proposed by the $i$th Principal is
$$\xi^{i,\star}=r_0-\sum_{j\neq i}y_0^j+\int_0^T\alpha_s^{i,\star}ds+\int_0^T\beta_s^{i,\star}\cdot dX_s, $$
with 
$$\alpha_s^{i,\star}=G\left(s,X_s, \sum_{j\neq i}\beta_s^{j}+ \beta_s^{i,\star}\right)-\sum_{j\neq i}\alpha_s^{j}. $$
\end{proposition}
\ms 

Since we solve each Principal's problem using an HJB equation, we naturally obtain a system of $N$ HJB equations. 
Similarly to \cite[Theorems 8.4 and 8.5]{Dockner}, one can derive from the system of equations a Nash equilibrium in the sense of Definition \ref{def:nash}.

We shall assume that there are solutions to PDEs \eqref{hjbi} for all $i=1,\cdots,N$, which also satisfy the equilibrium condition \eqref{cor:necessary}.
\ms

\begin{assumption}\label{firstorderN}For any $i\in \{1,\dots,N\}$ and $(t,x,y,p^i,\tilde p^i,\tilde q^i,r^i)\in [0,T]\times \mathbb R^N\times \mathbb R\times \mathbb R^N\times \mathbb R\times \mathbb R\times \mathbb R$ there exists $(\alpha^{i,\star}\left(t,x,y,(p^i,\tilde p^i,\tilde q^i,r^i)_{1\leq i\leq N})\right)_{1\leq i\leq N},\beta^{i,\star}\left(t,x,y,(p^i,\tilde p^i,\tilde q^i,r^i)_{1\leq i\leq N})\right)_{1\leq i\leq N}$ such that for any $1\leq i\leq N$
\begin{equation}\label{firstordercondition}
\begin{cases}
\displaystyle \lambda^i\left(t,x,y,p^i,\tilde p^i,\tilde q^i,r^i,\sum_{j\neq i} \beta^{j,\star}, \beta^{i,\star}\right)=0,\\
\displaystyle \sum_{i=1}^N \alpha^{i,\star}_t= G\left(t,x, \sum_{i=1}^N\beta^{i,\star}_t\right).
\end{cases}
\end{equation}
Moreover, we assume that the following system of fully coupled PDEs admits a solution $u=(u^i)_{1\le i\le N}$ such that $u^i\in C^{1,2,2}$ for all $1\le i\le N$:
\begin{equation}\label{hjbN}
\begin{cases}
\displaystyle \mathcal L\left(t,x,y,u^i,\nabla_x u^i,\partial_y u^i,\nabla_x^2 u^i,\partial^2_{yy}u^i,\partial^2_{x,y} u^i,S_{\alpha^{(-i)}},S_{\beta^{(-i)}},\beta^{i,\star} \right)= k_tu^i,\\[0.8em]
\displaystyle u^i(T,x,y)=U_{P^i}(\ell_i(x)-y),\; (t,x,y)\in [0,T)\times\mathbb R^N\times \mathbb R.
\end{cases} 
\end{equation}
with $\mathcal L:=\partial_t+H$, 
$$\beta^{i,\star}:= \beta^{i,\star}\left(t,x,y,(\nabla_x u^i,\partial_y u^i,\partial^2_{yy}u^i,\partial^2_{x,y} u^i)_{1\leq i\leq N}\right) $$
 $$S_{\alpha^{(-i)}}:= \sum_{j\neq i} \alpha^{j,\star}\left(t,x,y,(\nabla_x u^i,\partial_y u^i,\partial^2_{yy}u^i,\partial^2_{x,y} u^i)_{1\leq i\leq N}\right)$$ and 
 $$S_{\beta^{(-i)}}:=\sum_{j\neq i} \beta^{j,\star}\left(t,x,y,(\nabla_x u^i,\partial_y u^i,\partial^2_{yy}u^i,\partial^2_{x,y} u^i)_{1\leq i\leq N}\right).$$
\end{assumption}\vspace{0.5em}

\begin{theorem}\label{thm:principal:nash}
Let Assumption \ref{firstorderN} be true. Further, define for any $1\leq i\leq N$ 
 $$\xi^{i,\star}=y_i+\int_0^T\alpha_s^{i,\star}ds+\int_0^T\beta_s^{i,\star}\cdot dX_s,$$
  with $$\alpha^{i,\star}_t:= \alpha^{i,\star}\left(t,x,y,(\nabla_x u^i,\partial_y u^i,\partial^2_{yy}u^i,\partial^2_{x,y} u^i)_{1\leq i\leq N}\right)$$ and $$\beta^{i,\star}_t:= \beta^{i,\star}\left(t,x,y,(\nabla_x u^i,\partial_y u^i,\partial^2_{yy}u^i,\partial^2_{x,y} u^i)_{1\leq i\leq N}\right).$$
Then, the set of contract $(\xi^{i,\star})_{1\leq i\leq N}$ is an admissible Nash equilibrium.
\end{theorem}
\begin{proof}
The result follows from the same argument as \cite[Theorems 8.4 and 8.5]{Dockner}, thanks to Proposition \ref{prop:Pi}.
\end{proof}

\subsection{Special case for risk-neutral Principals}\label{sec:specialcase}

The main difficulty on the previous procedure of solving the Principals' problem is to check that Assumption \ref{firstorderN} holds true, \textit{i.e.} there exists a solution to the system \eqref{hjbN} of fully coupled HJB equations. It is obviously a tough question in general cases. In this section, we will show that if the Principals are all risk-neutral and there is no discount factor, i.e.
$$k\equiv 0\q \mbox{and}\q U_{P^i}(\ell_i(x)-y)=\ell_i(x)-y,\q (x,y)\in \mathbb R^N\times \mathbb R,$$
 the system of equations are solvable under reasonable assumptions.

Recall the value function of the $i$th Principal \eqref{pb:prncipaliyi}. Since $k\equiv 0$ and the Principal is risk-neutral, we have
\beaa
u^{i}(0,x;y^i) &:=& \sup_{\beta^i\in  \dbH^2_{BMO}}\mathbb E^{\nu^{\star}}\left[\ell_i(X_T)-Y_T^{y^i,\alpha^i,\beta^i}\right]\\
&=&  \sup_{\beta^i\in   \dbH^2_{BMO}}\mathbb E^{\nu^{\star}}\left[\ell_i(X_T)- y^i -\int_0^T \left(G\left(s,X_s, \bar \beta_s\right)-S_{\alpha_s^{(-i)}}\right) ds-\int_0^T \beta_s^{i}\cdot dX_s\right], \\
&=& - y^i +\sup_{\beta^i\in   \dbH^2_{BMO}}\mathbb E^{\nu^{\star}}\left[\ell_i(X_T)-\int_0^T \left(G\left(s,X_s, \bar \beta_s\right)-S_{\alpha_s^{(-i)}}+ \beta_s^{i}\cdot b\big(s,X_s,\nu^*_s(X_s, \bar\b_s)\big) \right)ds\right],
\eeaa
where $\bar\b : = \sum_{i=1}^N \b^i$. Since we are only interested in finding a particular Nash equilibrium, we may assume 
\beaa
\a^i_s = \frac{G(s,X_s, \bar \beta_s)}{N},
\eeaa
so that the second line of \eqref{firstordercondition} is satisfied and the control problem above is further simplified to be:
\beaa
u^{i}(0,x;y^i) =  - y^i +\sup_{\beta^i\in   \dbH^2_{BMO}}\mathbb E^{\nu^{\star}}\left[\ell_i(X_T)-\int_0^T \left(\frac{G\left(s,X_s, \bar \beta_s\right)}{N}+ \beta_s^{i}\cdot b\big(s,X_s,\nu^*_s(X_s, \bar\b_s)\big) \right)ds\right].
\eeaa
As we can see, the stochastic control problem here has no longer the state variable $Y$ as in the general case. Thus, we choose $y^i$ so that $\sum_{i=1}^N y^i=r_0 $ so that the corresponding HJB equation is semilinear as follow:
\begin{multline}\label{semilinearHJB}
\pa_t u^i + \frac12\text{Tr}( \Si_t\Si_t ^\top \nabla^2_{x}u^i)  \\
+\sup_{\b}\Big\{(\nabla_x  u^i -\b)  \cd b\big(t,x,\nu^*_t(x,\bar\b_t )\big) 
- \frac{G(t,x,\bar\b_t)}{N}\Big\}=0, \q u^i(T,x) =\ell_i(x).
\end{multline}
Further, the first order condition for the supremum in \eqref{semilinearHJB} being reached reads:
\bea
&\beta^{i,\star}_t =\nabla_x u^i(t,x)-\frac{1}{N}\Phi(t,x,\bar\b) &\label{FOC:ith} \\
 &\mbox{with}\q \Phi(t,x,\bar\b):=M_{\bar\beta}^{-1}\left(\nabla_{\bar\beta}G(t,x,\bar\beta_t)+ N b\big(t,x,\nu^\star(t,x,\bar\b_t)\big)\right), &\label{def:Phi}
\eea
where $M_{\bar\beta}:=\nabla_\nu b(t,x,\nu^\star(t,x,\bar\beta_t))\nabla_{z}\nu^\star(t,x,\bar \beta_t)$. In order that \eqref{FOC:ith} makes sense and for the upcoming reasoning, we introduce the following assumption. 
\ms

\begin{assumption}\label{assum:semi}
Assume that 
\begin{itemize}
\item $\nu^*:(t,x,z)\mapsto \nu^*(t,x,z)$ is continuously differentiable in $z$;
\item $M_{\bar\beta}$ is invertible for all $(t,x,\bar\b)$;
\item the function $Id + \Phi(t,x,\cd)$ is invertible, and its inverse function is denoted by $\f(t,x,\cd)$.
\end{itemize}
Further, define
\bea\label{semi-H}
\Hc(t,x,z):=
 \big(z- \f(t,x,z)\big)\cd b\big(t,x,\nu^*(t,x,\f(t,x,z))\big) - G\big(t,x,\f(t,x,z)\big),
\eea
and assume that
\begin{itemize}
\item for some $\g,\a\in (0,1)$, $L\, :\, x \mapsto \sum_{i=1}^N\ell_i(x) \in H_\g(\dbR^N) $ (the H\"older space),  $\Si_t \Si_t^\top \in H_1([0,T])$  and $\Hc\in H_\a(K)$ for any bounded subset of $[0,T]\times \dbR^N\times\dbR^N$;
\item $\Hc (t,x,z) = O(|z|^2)$ for all $(t,x)\in [0,T]\times\dbR^N$.
\end{itemize}
\end{assumption}
\ms

\begin{remark}
The first part of Assumption \ref{assum:semi} is structural in our approach, while the second part is a sufficient condition for the solvability of the PDE we will encounter in the following argument. The readers may find alternative conditions in different contexts.
\end{remark}

\ms
\begin{theorem}\label{thm:semi}
Let Assumption \ref{assum:si b}, \ref{assum:c} and \ref{assum:semi} hold true. Then, the system of equations composed of the equations as \eqref{semilinearHJB}, for all $i=1,\cdots,N$, admits a classical solution.
\end{theorem}

\begin{remark}
The main idea of proving the result above is to study the `aggregated' equation. Let us do the following intuitive analysis.
Summing up the equations \eqref{FOC:ith} for all $1\le i\le N$, we formally have
\beaa
\bar\b = \nabla_x V-\Phi(\cd,\bar\b),
\eeaa
where $V =\sum_{i=1}^N u^i$, and thus $\bar\b = \f (\cd,\nabla_x V)$.
Using the optimal $\b^{i,\star}$ in \eqref{FOC:ith},  we can rewrite \eqref{semilinearHJB} for all $1\le i\le N$ as
\begin{multline}\label{eq:uimod}
\pa_t u^i + \frac12\text{Tr}( \Si_t\Si_t ^\top \nabla^2_{x}u^i)  \\
+(\nabla_x  u^i -\b^{i,\star}_t)  \cd b\big(t,x,\nu^*_t(x,\bar\b_t )\big) 
- \frac{G(t,x,\bar\b_t)}{N}=0, \q u^i(T,x) =\ell_i(x).
\end{multline}
Sum up the $N$ equations, and we obtain the so-called `aggregated' equation:
\begin{equation}\label{PDE:aggreg:proof}
\begin{cases}
\displaystyle \partial_t V+\frac12 \text{Tr}\left(\Sigma_t\Sigma^\top_t \nabla^2_xV\right)+ \Hc(\cd, \nabla_x V)=0,\; (t,x)\in [0,T)\times \mathbb R^N ,\\
\displaystyle V(T,x)= L(x),
\end{cases} 
\end{equation}
where $\Hc$ is defined in \eqref{semi-H}.
Under Assumption \ref{assum:semi}, the equation above admits a solution $V\in H^{(-1-\g)}_{2+\a}$ (see e.g. Theorem 12.16, pp. 315, \cite{Lieberman}). Note that \eqref{PDE:aggreg:proof} does not necessarily have an unique solution, but it does not bother the following argument.
\end{remark}

\ms
\noindent{\it Proof of Theorem \ref{thm:semi}}\q
In this proof, we shall construct a solution to the system of equations, using one solution $V$ to \eqref{PDE:aggreg:proof}. Note that we have the following representation for $V$:
\bea\label{Vrep}
V(t,x) = \dbE^\dbP \Big[L(X_T) +\int_t^T \Hc\big(s,X_s, \nabla_x V(s,X_s)\big)ds \big| X_t =x\Big].
\eea
Now define for all $1\le i\le N$:
\beaa
\pa_t \tilde u^i + \frac12\text{Tr}( \Si_t\Si_t ^\top \nabla^2_{x} \tilde u^i) +\frac{1}{N}\Hc(\cd,\nabla_x V)=0, \q \tilde u^i(T,x) =\ell_i(x).
\eeaa
The equation above is a heat equation, so it admits a classical solution and 
\beaa
\tilde u^i(t,x) = \dbE^\dbP \Big[\ell_i(X_T) +\int_t^T \frac{1}{N}\Hc(s,X_s, \nabla_x V(s,X_s)) ds \big| X_t =x\Big].
\eeaa
Together with \eqref{Vrep}, it is clear that $V = \sum_{i=1}^N \tilde u^i$ and thus 
\bea\label{Vutilde}
\nabla_x V =  \sum_{i=1}^N \nabla_x \tilde u^i.
\eea
Now define 
\bea\label{veri:FOC}
\bar\b := \f(\cd, \nabla_x V),\q\q \b^i := \nabla_x \tilde u^i - \frac{1}{N}\Phi(\cd, \bar\b).
\eea
It follows from \eqref{Vutilde} that $\bar\b = \sum_{i=1}^N \b^i$. Also note that
\beaa
\frac{1}{N}\Hc(\cd,\nabla_x V) &=& \frac{1}{N}\Phi(\cd, \f(\cd, \nabla_x V))  \cd b\big(\cd,\nu^*(\cd,\f(\cd, \nabla_x V))\big) 
- \frac{G(t,x,\f(\cd, \nabla_x V))}{N} \\
& = & \big(\nabla_x \tilde u^i -\b^i \big)  \cd b\big(\cd,\nu^*(\cd,\bar\b)\big) - \frac{G(t,x,\bar\b)}{N}.
\eeaa
Therefore $\tilde u^i$ is a solution to \eqref{eq:uimod}. Together with \eqref{veri:FOC}, we conclude that $\tilde u^i$ is a classical solution to \eqref{semilinearHJB}.
\qed

\section{Application to a model with two correlated Principals with appetence parameters}\label{section:2principaux} 

\subsection{The bi-Principals model}
Let $\nu:= (\nu^1, \nu^2)^\top$ an admissible effort. Let $W$ be a $2$-dimensional Brownian motion and $\Sigma\in \mathbb R^{2,2}$ be an invertible matrix. We assume that
\begin{itemize}
\item there is no discount factor, \textit{i.e.} $k=0$.

\item both Principals are risk-neutral, that is, $U_{P^1}=U_{P^2}=I$, the identity function.

\item the drift $b$ is a linear function of the effort such that $b(t,x,\nu):= K\nu$, where $K$ is a diagonal matrix with coefficients $k_1,k_2$ on the diagonal which represents the efficiency of the Agent with project $1$ and $2$. 

\item the map $c$ is the classical quadratic cost function defined by $c(t,x,\nu)=\frac{\|\nu_t\|^2}2$.

\item $\ell_i(x):=(1+\gamma_i)x_i-\gamma_i x_j,\; 1\leq i\neq j\leq 2$ with appetence parameters $\gamma_1,\gamma_2 \in [0,1]$, and denote $\Gamma:= (1+\gamma_1-\gamma_2, 1+\gamma_2-\gamma_1)^\top$.
\end{itemize}
Clearly, this model is an example of the special case we discussed in Section \ref{sec:specialcase}.

 It follows from \eqref{vstar} that we have the following relation between the optimal effort of the Agent and the volatility coefficients of the contracts:
\beaa
\nu^\star= K(\beta^1+\beta^2).
\eeaa

Further, it follows from \eqref{pb:prncipaliyi} and \eqref{Ydynamic} that the value functions read:
\begin{multline*}
U_0^{P_i}(x) 
= \sup_{\beta^i}\mathbb E^{\nu^{\star}}\Big[\ell(X_T)- \int_0^T \Big( \frac{R_A}2\|\Sigma^\top (\beta^{1}_s+\beta^{2}_s) \|^2-\frac{\| K(\beta^{1}_s+\beta^{2}_s)\|^2}2 -S_{\alpha_s^{(-i)}} + K^2(\beta^1_s+\beta^2_s)\cdot \beta^{i}_s \Big) ds \Big].
\end{multline*}

We have the following corollary from Theorem \ref{thm:semi}.
\ms

\begin{corollary}\label{lemmaBiprincipaux} There exist $(\alpha^{1,\star},\beta^{1,\star})$ and $(\alpha^{2,\star},\beta^{2,\star})$ such that  the following system of PDEs admits a  solution $(v^1,v^2)$ in $\mathcal C^{1,2}([0,T]\times \mathbb R^2)$
 \begin{equation}\label{HJB21}
\begin{cases}
\displaystyle -\partial_t v^1(t,x)-\left\{\nabla_x v^1\cdot K^2(\beta^{1,\star}_t+\beta^{2,\star}_t) +\frac{1}2\text{Tr}\left( \Sigma\Sigma^\top\nabla^2_x v^1\right)- \frac{R_A}2\|\Sigma^\top (\beta^{1,\star}_t+\beta^{2,\star}_t) \|^2\right.\\
\displaystyle \hspace{9em}\left. +\frac{\| K(\beta^{1,\star}_t+\beta^{2,\star}_t)\|^2}2+\alpha_t^{2,\star}-K^2(\beta^{1,\star}_t+\beta^{2,\star}_t)\cdot \beta_t^{1,\star}\right\}=0\\
\displaystyle v^1(T,x)=(1+\gamma_1)x_1-\gamma_1 x_2,
\end{cases} 
\end{equation}
and 
\begin{equation}\label{HJB22}
\begin{cases}
\displaystyle -\partial_t v^2(t,x)-\left\{\nabla_x v^2\cdot K^2(\beta^{1,\star}_t+\beta^{2,\star}_t) +\frac{1}2\text{Tr}\left( \Sigma\Sigma^\top \nabla^2_x v^2 \right)- \frac{R_A}2\|\Sigma^\top (\beta^{1,\star}_t+\beta^{2,\star}_t)\|^2\right.\\
\displaystyle \hspace{9em}\left. +\frac{\| K(\beta^{1,\star}_t+\beta^{2,\star}_t)\|^2}2+\alpha_t^{1,\star}-K^2(\beta^{1,\star}_t+\beta^{2,\star}_t)\cdot \beta_t^{2,\star}\right\}=0\\
\displaystyle v^2(T,x)=(1+\gamma_2)x_2-\gamma_2 x_1,
\end{cases} 
\end{equation}
and that the first order conditions \eqref{firstordercondition} hold true, that is,
\begin{equation}\label{b1b2}
\begin{cases}
\beta^{1,\star} = \nabla_x v^1 - K^{-2} R_A\Sigma\Sigma^\top(\beta^{1,\star}+\beta^{2,\star})\\
\beta^{2,\star} = \nabla_x v^2 - K^{-2} R_A\Sigma\Sigma^\top(\b^{1,\star} + \beta^{2,\star})\\
\alpha^{1,\star}+\alpha^{2,\star} = \frac{R_A}{2}\|\Sigma^\top (\beta^{1,\star}_t+\beta^{2,\star}_t)\|^2 - \frac{K^2\|\beta^{1,\star}_t+\beta^{2,\star}_t\|^2}{2}.
\end{cases} 
\end{equation}
\end{corollary}

In the rest of this section, we are going to compute the explicit solutions to the system of equations \eqref{HJB21}, \eqref{HJB22}. We learn from \eqref{b1b2} that the function $\Phi$ as \eqref{def:Phi} in this case is:
\beaa
\Phi(\b) = 2K^{-2} R_A\Sigma\Sigma^\top \b.
\eeaa
Further, we have
\beaa
\f(\b) = (Id + \Phi)^{-1}(\b) = \big(I + 2K^{-2} R_A\Sigma\Sigma^\top\big)^{-1} \b =: M\b
\eeaa
Therefore, the `aggregated equation' in this case reads:
\begin{equation}\label{HJBsomme}
\begin{cases}
\displaystyle -\partial_t V(t,x) -\frac{1}2\text{Tr}\left(\Sigma\Sigma^\top \nabla^2_x V\right) - \nabla_x V \cdot K^2 M\nabla_x V+ \frac{R_A}2\|\Sigma^\top M \nabla_x V\|^2  + \frac{\| KM \nabla_x V\|^2}2=0\\
\displaystyle V(T,x)=x\cdot \Gamma,
\end{cases} 
\end{equation}
\noindent It is easy to observe that PDE \eqref{HJBsomme} has the unique smooth solution 
\bea\label{expsol}
&V(t,x)=x\cdot \Gamma +\lambda (T-t),&\\
&\mbox{with}\quad \lambda = \Gamma\cdot K^2M\Gamma - \frac{R_A}{2}\|\Sigma^\top M\Gamma\|^2 -\frac{\|KM\Gamma\|^2}{2}.&\notag
\eea
 In particular, $\nabla_x V = \Gamma$, and thus
\bea\label{nuoptimal}
\beta^{1,\star}+ \beta^{2,\star}  = \f(\nabla_x V)= M\Gamma \q
\mbox{and}\q \nu^*=KM\G.
\eea
Together with the first order condition \eqref{b1b2}, we  have
\begin{equation}\label{b1b2new}
\begin{cases}
\beta^{1,\star} = \nabla_x v^1 - K^{-2} R_A\Sigma\Sigma^\top M\Gamma,\\
\beta^{2,\star} =  \nabla_x v^2 - K^{-2} R_A\Sigma\Sigma^\top M\Gamma.\\
\end{cases} 
\end{equation}
Therefore, assuming $\a^{1,\star} = \a^{2,\star}$, we know that the solution to the equation \eqref{HJB21} satisfies:
\beaa
\begin{cases}
\displaystyle -\partial_t v^1(t,x)-\frac{1}2\text{Tr}\left( \Sigma\Sigma^\top\nabla^2_x v^1\right) 
- \frac12 \G \cdot K^2 M\G+ \frac{R_A}4\|\Sigma^\top M \G\|^2  + \frac{\| KM \G\|^2}4 =0\\
\displaystyle v^1(T,x)=(1+\gamma_1)x_1-\gamma_1 x_2,
\end{cases} 
\eeaa
It is easy to compute the explicit solution to the equation above:
\beaa
& v^1(t,x):= \tilde\lambda (T-t) + (1+\gamma_1,-\gamma_1)x &\\
& \mbox{with}\q \tilde \lambda : = \Gamma \cdot K^2 M\Gamma - \frac{5R_A}{4}\|\Sigma^\top M\Gamma\|^2 -\frac{3}{4}\|KM\Gamma\|^2.  &
\eeaa
Similarly, we compute the solution to the equation \eqref{HJB22} as
\beaa
v^2(t,x):= \tilde\lambda (T-t) + (-\gamma_2, 1+\gamma_2)x.
\eeaa


\subsection{Comparison with the model with the aggregated offer}\label{section:casPrincipauxFullyCooperatif} 

In this section, we compare the competitive common agency example studied in the previous section with the model in which the Principals can be aggregated, that is, the problem can be reduced to one single Principal-Agent model  as mentioned at the end of Section \ref{soussection:principalequilibrium:intro}. In this case, the Agent problem can be solved by considering the same class of contracts, that is, there exists $(Y_0,Z)\in [R_0,+\infty)\times \mathbb H^2_{BMO}(\mathbb R^2)$ such that
\begin{align*}
\xi&= Y_0+\int_0^T Z_s \cdot \Sigma dW_s+\int_0^T\left(\frac{R_A}2\|\Sigma^\top Z_s\|^2+\frac{\|KZ_s\|^2}2-KZ_s\cdot Z_s \right)ds\\
&=Y_0+\int_0^T Z_s \cdot \Sigma dW^\star_s+\int_0^T \left(\frac{R_A}2\|\Sigma^\top Z_s\|^2 +\frac{\|KZ_s\|^2}2 \right) ds,
\end{align*}
where $W^\star$ is a $\mathbb P^{\nu^\star}$-Brownian motion.
The Principal manages to do the following optimization:
\begin{align*}
U_0&=\sup_Z \mathbb E^\star\left[ (1+\gamma_1-\gamma_2)X_T^1+ (1+\gamma_2-\gamma_1)X_T^2-\xi\right]
\end{align*}
As before, we can calculate the optimal control:
 $$Z^\star=M_a\Gamma,\quad\mbox{with}\quad M_a:= \left(R_A\Sigma\Sigma^\top
+K^2\right)^{-1}K^2.$$
Further, the optimal effort of the Agent should be
$$\nu_{Pf}^\star=K M_a\Gamma. $$
By comparing $\nu_{Pf}^\star$ and  $\nu^\star$ in \eqref{nuoptimal}, we immediately have the following conclusion.
\vspace{0.5em}

\begin{proposition}\label{prop:bernheim}
The effort of the common Agent in the competitive model coincides with the effort of the common Agent in the aggregated model if and only if the Agent is risk-neutral, that is, $R_A=0$.
\end{proposition}

\paragraph{Link with the results of \cite{Bernheim} for the discrete-case model.}
Recall that if the Agent is risk-neutral, it is well-known that in the single Principal-Agent problem, the effort in the first best case coincides with that in the second best case (see for instance \cite[Proposition 4.1]{laffont2009theory}). Indeed, as an extension of it in our particular model, if one computes the first best effort, one has to solve
$$U_0^{FB}=\sup_{\nu,\; \xi} \mathbb E^\nu\left[ X_T\cdot \Gamma -\xi\cdot \mathbf 1_2 -\rho e^{-R_A(\xi\cdot \mathbf 1_2 -\int_0^T \frac{\|\nu_s\|^2}2 ds)}\right], $$ with a Lagrange multiplier $\rho>0$ ensuring that the Agent receives his utility reservation $R_0<0$. In this case, by using G\^ateaux derivative to characterize the optimal $\xi$ (see for instance the method used in \cite{EliePossamai}) one gets after an easy computation the following optimizers
$$\nu^{\star,FB}_t:=K\Gamma, $$
$$\xi^\star\cdot \mathbf 1_2=K\Gamma-\frac{1}{R_A}\log\left(-R_0 \right). $$
Therefore, we have proved that the 'second best' effort of a risk-neutral Agent in the competitive model coincides with the 'first best' effort. This is exactly an extension of the result in \cite{Bernheim} in the discrete-case. 

\paragraph{Comparison of optimal remunerations.}
In the common agency model we have seen that the remuneration given to the Agent is
$$\xi^\star_1+ \xi^\star_2=R_0+T\underbrace{\left( \frac{R_A}2 \|  \Sigma^\top M \Gamma \|^2+ \frac{\| KM \Gamma\|^2}2\right)}_{=:\delta}+\int_0^T M\Gamma\cdot \Sigma dW_s^\star. $$
In the aggregated model, we recall that
$$\xi^\star:=R_0+T\underbrace{\left( \frac{R_A}2 \|  \Sigma^\top M_a \Gamma \|^2+ \frac{\| KM_a \Gamma\|^2}2\right)}_{=:\delta_a}+ \int_0^T M_a\Gamma\cdot \Sigma dW_s^\star $$

\noindent Let $\Sigma:=I_2$, then, we get 
$$M=\begin{pmatrix} \frac{k_1^2}{2R_A+k_1^2}&0 \\ 0& \frac{k_2^2}{2R_A+k_2^2} \end{pmatrix} $$ 
and
$$M_a=\begin{pmatrix} \frac{k_1^2}{R_A+k_1^2}&0 \\ 0& \frac{k_2^2}{R_A+k_2^2} \end{pmatrix} $$

\noindent In this particular model, we have $\delta_a\geq \delta$, \textit{i.e.} the non-risk part of the remuneration is higher for the Agent if he is employed by the aggregated firm (the parent firm), comparing to the case he is hired by two different firms. This is curious, because intuitively the Agent should receive higher return if the firms compete with each other. However, one must note that the Agent works more ($\|\nu_{Pf}^\star\|_1\ge \|\nu^\star\|_1$) in the aggregated model which could explained this effect.

%

\subsection{Impact of the appetence, efficiency and correlation parameters}\label{section:impactEffortAmbitions}

Let $\rho\in [-1,1]$ be a correlation parameter such that $$\Sigma:=\begin{pmatrix} 1&0\\ \rho&\sqrt{1-\rho^2}\end{pmatrix}, \; K:=\begin{pmatrix}
k_1&0\\
0&k_2
\end{pmatrix}
$$
In this case, we get after a (tedious but easy) computation
\begin{align*}
\nu^{1,\star}&= \frac{2R_Ak_1\left((1+\gamma_2-\gamma_1)k_2^2\rho -(1+\gamma_1-\gamma_2)k_1^2\right)-k_1^3k_2^2(1+\gamma_1-\gamma_2)}{2R_A^2(\rho^2-1)-2R_A(k_1^2+k_2^2)-k_1^2k_2^2}\\
\nu^{2,\star}&=  \frac{2R_Ak_2\left((1+\gamma_1-\gamma_2)k_1^2\rho -(1+\gamma_2-\gamma_1)k_2^2\right)-k_1^2k_2^3(1+\gamma_2-\gamma_1)}{2R_A^2(\rho^2-1)-2R_A(k_1^2+k_2^2)-k_1^2k_2^2}\\
\end{align*}
\paragraph{Impact of the correlation and the appetence parameters.}
Assume that $k:=k^1=k^2$.
In this case
\begin{align*}
\nu^{1,\star}(\rho)&:= \frac{2R_Ak^3\left((1+\gamma_2-\gamma_1)\rho -(1+\gamma_1-\gamma_2)\right)-k^5(1+\gamma_1-\gamma_2)}{2R_A^2(\rho^2-1)-4R_Ak^2-k^4}\\
\nu^{2,\star}(\rho)&:=  \frac{2R_Ak^3\left((1+\gamma_1-\gamma_2)\rho -(1+\gamma_2-\gamma_1)\right)-k^5(1+\gamma_2-\gamma_1)}{2R_A^2(\rho^2-1)-4R_Ak^2-k^4}\\
\end{align*}
Note that 
\begin{align*}
\nu^{1,\star}(\rho)-\nu^{2,\star}(\rho)&= \frac{4R_Ak^3\left((\gamma_2-\gamma_1)(\rho+1) \right)-2k^5(\gamma_1-\gamma_2)}{2R_A^2(\rho^2-1)-4R_Ak^2-k^4}\\
&=(\gamma_1-\gamma_2)  \frac{-4R_Ak^3(\rho+1)-2k^5}{2R_A^2(\rho^2-1)-4R_Ak^2-k^4}.
\end{align*}
Thus as soon as $\gamma_1>\gamma_2$, we get $\nu^{1,\star}(\rho)>\nu^{2,\star}(\rho)$ so that the Agent works more for the more ambitious principal. Moreover, we have 
$$d(\rho):=\frac{\nu^{1,\star}}{\overline{\nu^{\star}}}(\rho) - \frac{\nu^{2,\star}}{\overline{\nu^{\star}}}(\rho)
=(\gamma_1-\gamma_2)\frac{2R_Ak^3(1+\rho)+k^5}{2R_Ak^3(1-\rho)+k^5},
$$
with $\overline{\nu^{\star}}:= \nu^{1,\star}+ \nu^{2,\star}$. 

\noindent By noticing that the function $\rho\in [-1,1)\longmapsto d(\rho)$ is convex and increasing, we deduce that the difference between the proportion of effort given by the Agent with the more ambitious principal and the proportion of effort given by the Agent with the less ambitious principal increases with the parameter $\rho$. In other words, the more the projects are correlated, the more the proportion of effort given by the Agent is impacted by the difference of the ambition parameters of each Principal. Besides, the convexity shows that the sensibility of the proportion of work given by the Agent for each principal increases with the correlation parameter. The more the projects are correlated, the more a little variation of the correlation has a big effect on the difference of proportion of effort given by the Agent.
\vspace{0.5em} We package these economical results in the following proposition

\begin{proposition}\label{prop:faitstylized} The following stylized facts has been proved
\begin{itemize}
\item The non-risk part of the remuneration is higher for the Agent if he is employed by the aggregated firm (the parent firm), comparing to the case he is hired by two different firms.
\item Trivially, by considering equals performance parameters $k^i$, the Agent works more for the more ambitious Principal.
\item The more the projects are correlated, the more the proportion of effort given by the Agent is impacted by the difference of the ambition parameters of each Principal. Moreover, the sensibility of the proportion of work given by the Agent for each principal increases with the correlation parameter.
\end{itemize}
\end{proposition} 

\paragraph{Risk-Neutral Agent.}
Assume that $\Sigma=\text{Id}_2$, $R_A=0$. In this case, $M=\text{Id}_2$ and $\nu^\star= K \Gamma.$ In other words we have 
\begin{align*}
\nu^{1,\star}&= k_1(1+\gamma_1-\gamma_2)\\
\nu^{2,\star}&= k_2(1+\gamma_2-\gamma_1).
\end{align*}
The Agent works more for the Principal $1$ if and only if
\begin{equation}\label{condition12}
k_1(1+\gamma_1-\gamma_2)>  k_2(1+\gamma_2-\gamma_1).
\end{equation}
Let now $x:= \gamma_2-\gamma_1$. 
\begin{itemize}
\item First, we note some intuitive results. If the two Principals have the same ambition parameters, (\textit{i.e.} $\gamma_1=\gamma_2$), the Agent prefers working with the Principal with whom he is more efficient. Similarly,  if the two Principals have the efficiency parameters, (\textit{i.e.} $k_1=k_2$), the Agent works more for the more ambitious Principal.
\item Assume now for instance that $\gamma_1=0,\; \gamma_2=1$. Then condition \eqref{condition12} is never satisfies and the Agent does not work for Principal $1$ whatever is her performance. It means that the Agent does not work for the Principal indifferent to the competition, while the other Principal is very competitive.

\item Assume now that $\gamma_1\in [0,1]$ and $\gamma_2\in [0,1)$. Denote $x:=\gamma_2-\gamma_1\in [-1,1)$. The domain $x\leq 0$ coincides with the situation where the Principal $1$ is more ambitious that the Principal $2$ (and conversely, Principal $1$ is less ambitious than Principal $2$ when $x\geq 0$). Condition \eqref{condition12} can be rewritten
$$\frac{k_1}{k_2}> \frac{1+x}{1-x},\; x\in [-1,1). $$
Let $f(x):=  \frac{1+x}{1-x},\; x\in [-1,1)$, then $f$ is clearly increasing and convex.\vspace{0.3em}

Let us provide an interpretation of the growth of $f$. When $x$ is bigger than $0$, we notice that $f(x)$ is bigger than $1$, which shows that a lack of ambition the Principal $1$, \textit{i.e.} $\gamma_1<\gamma_2$, can be balanced with the efficiency parameter of her since it exists a domain of $x\geq 0$,  \textit{i.e.} for which the Principal 1 is less ambitious than the Principal 2 but such that the Agent works for the less ambitious Principal 1.\vspace{0.3em}

Turn now to the convexity of $f$. This phenomenon is quite interesting since it suggests that when a Principal is clearly more ambitious than the other, for instance $\gamma_1<<\gamma_2$, a little modification of the ambition parameters leads to a high variation of the quotient between the effort of the Agent for the Principal 1 and for the Principal 2. In other words, if the Principal 1 is clearly less ambitious than the Principal 2, derived lightly from this state leads to a big difference of the efforts provides by the Agent to manage the project of the Principal 1 and the Principal 2. If for instance the Principal 1 increases a little her ambition, the Agent will manage quite more her project. There is a kind of leverage effect between a deviation of an initial ambition parameter and the quantity of work provides to the Agent, when the two Principal have very different behaviours concerning their relative performances. Besides, a little modification of the ambition parameter of a Principal significantly more indifferent than the other with fixed ambition parameter, increases or decreases meaningfully the range of possible efficiency parameters to have an Agent more devoted to her project.  \end{itemize}\vspace{0.5em}
\begin{remark}
As explained in the section above with Proposition \ref{prop:bernheim}, this case can be in fact reduced to the classical second best case (with one Principal), by aggregating the Principals. Thus, we recover the classical interpretations by adding efficiency parameter $\gamma_i$ for project $i$ with $i=1,2$.
\end{remark}

\section{Conclusion}

In this paper, we show that in the common agency problem, the value functions of the Principals should be the solutions of a system of HJB equations, that we solve when the Principals are risk-neutral. Meanwhile, the coefficients of the system should satisfy the Nash equilibrium conditions. In particular, we study a model in which two Principals hire a common Agent, and we obtain one of the main results in \cite{Bernheim}, that is, the outcome with two non-cooperative employers and that with an aggregated employer coincide only when the Agent is risk-neutral.\vspace{0.3em}

\noindent At the opposite of common agency problem, other works have investigated a competition problem between two principals who want the exclusive service of an agent. In 1976, Rothschild and Stiglitz \cite{rothschild1976equilibrium} have studied insurance markets in which principals are identical and compete for a single agent. It has then been extended by Biglaiser and Mezzetti in \cite{BiglaiserMezzetti} in both moral hazard and adverse selection cases. This problem is not considered here, because we allows the agents to be remunerated by all the Principals, and it will be left for future researches. 

\section*{Acknowledgments}
We would like to thank two anonymous referees and the associated editors for their very relevant suggestions leading to a significant improvement of our paper. We are very grateful to Johannes Muhle-Karbe for very instructive discussions on common agency theory and his  suggestions and remarks in particular on the interpretations of some of our results. We also would like to thank Dylan Possama\"i and Cl\'emence Alasseur for their valuable advices.


\newcommand{\etalchar}[1]{$^{#1}$}

\end{document}